\documentclass{article}[11pt]
\usepackage{graphicx,amssymb,mathrsfs,amsmath,latexsym,amsfonts,amsthm}
\usepackage{amsmath,amssymb,amsthm}
\usepackage{latexsym}
\usepackage{amssymb}
\usepackage{stmaryrd}
\usepackage{graphicx}
\usepackage{psfrag}
\usepackage{xcolor}
\usepackage{enumerate}

\makeatletter 
\@addtoreset{equation}{section}
\makeatother
\newtheorem{theorem}{Theorem}[section]
\newtheorem{example}[theorem]{Example}
\newtheorem{remark}[theorem]{Remark}
\newtheorem{corollary}[theorem]{Corollary}
\newtheorem{proposition}[theorem]{Proposition}

\def\adots{\mathinner{\mkern2mu\raise0pt\hbox{.}  
\mkern2mu\raise4pt\hbox{.}\mkern1mu
\raise7pt\vbox{\kern7pt\hbox{.}}\mkern1mu}}

\def\IC{{\mathbb C}}

\def\cE{{\mathcal E}}
\def\cS{{\mathcal S}}
\def\cV{{\mathcal V}}

\def\rank{{\rm rank\,}}
\def\Ker{{\rm Ker}}
\def\ker{{\rm ker}}
\def\vect{{\rm vec}}

\title{Linear rank preservers of tensor products of rank one matrices}
 \author{Zejun Huang, \quad Shiyu Shi,\quad and \quad Nung-Sing Sze}
\date{}

\begin{document}
\maketitle

\noindent {\bf Abstract}. Let $n_1,\ldots,n_k $ be  integers larger than or equal to 2.  We
characterize  linear maps  $\phi: M_{n_1\cdots n_k}\rightarrow M_{n_1\cdots n_k}$ such that
 \begin{equation*}
\rank(\phi(A_1\otimes \cdots \otimes A_k))=1\quad\hbox{whenever}\quad\rank (A_1\otimes \cdots \otimes A_k)=1 \quad \hbox{for all}\quad  A_i \in M_{n_i},\, i = 1,\dots,k.
\end{equation*}
Applying this result, we extend two recent results on linear maps that preserving the rank of special classes of matrices.



\section{Introduction and statement of main results}

Let $n\ge 2$ be positive integers. Denote by  $M_{  n}$  the set of $n\times n$ complex matrices  and $\IC^n$ the set of complex column vectors with $n$ components.
Linear preserver problems concern  the study of linear maps on matrices or operators with some special properties, which has a long history.  In 1897, Frobenius \cite{F} showed that a linear operator
$\det(\phi(A)) = \det(A)$ for all $A\in M_n$
if and only if  there are $M, N\in M_n$ with $\det(MN) = 1$ such that
$\phi$ has the form
\begin{equation*} \label{standard}
A \mapsto MAN \quad \hbox{ or } \quad A \mapsto MA^tN.
\end{equation*}
Since then, lots of linear preservers have been characterized, see \cite{FHLS,LP} and their references. In particular, Marcus and Moyls \cite{MM} determined  linear maps   that send  rank one matrices to rank one matrices, which have the form $A\mapsto MAN$ or $A\mapsto MA^TN$ for some nonsingular matrices $M$ and $N$.

Recently, linear maps that preserve certain properties of tensor products are studied.
The {\it tensor product } ({\it Kronecker product}) of  two matrices $A\in M_m$ and $B\in M_n$ is defined to be   $A\otimes B= \left[ a_{ij}B \right] $, which is in $M_{mn}$.
In \cite{FHLS}, the authors determined  linear maps on Hermitian matrices that leave the spectral radius of all tensor products invariant.  In \cite{FHLPS,FHLS2,FHLS3,LPS} the authors determine  linear maps on $M_{mn}$ that preserve Ky Fan norms, Shattern norms, numerical radius, $k$-numerical range, product numerical range of all matrices
of the form $A \otimes B$ with $A\in M_m$ and $B\in M_n$.
Notice that the set of matrices of tensor product form shares only a very small portion in $M_{mn}$
and the sum of two tensor products is in general no longer a tensor product form.
Therefore, such linear preserver problems are more challenging than the traditional problems.
In some of the above mentioned papers, the authors have also extended their results to multipartite system,
i.e., matrices of the form $A_1\otimes \cdots \otimes A_k$ with $k \ge 2$.

In the literature, rank preserver problem is known to be one of the fundamental problems in this subject
as many other preserver problems can be deduced to rank preserver problems.
For example, the result of Marcus and Moyls \cite{MM} on linear rank one preservers
have been applied in  many other preserver results. More discussion can be found in \cite{GLS}.
Let $n_1,\ldots, n_k$ be positive integers of at least two.
In \cite{ZXF}, Zheng, Xu and Fo\v sner showed that a linear map   $\phi: M_{n_1\cdots n_k}\rightarrow M_{n_1\cdots n_k}$ satisfies
\begin{equation}\label{eq121}
 \rank \phi(A_1\otimes \cdots \otimes A_k)=\rank (A_1\otimes \cdots \otimes A_k)
\quad \hbox{for all } A_i\in M_{n_i},\ i=1,\ldots,k
 \end{equation}
if and only if $\phi$ has the form
 \begin{equation}\label{eqa11}
 \phi(A_1\otimes \cdots \otimes A_k)=M(\psi_1(A_1)\otimes \cdots\otimes \psi_k(A_k))N
 \end{equation}
 where $M,N\in M_{n_1\cdots n_k}$ are nonsingular and $\psi_i$, $i = 1,\dots,k$, is either the identity map or the transpose map.
 Their proof was done by induction on $k$ with some smart argument on the rank of sum of certain matrices.
 The same authors also considered in \cite{XZF} the injective maps on the space of Hermitian matrices satisfying (\ref{eq121}) for rank one matrices only.
By using a structure theorem of Westwick \cite{RW},
Lim \cite{LIM} improved the result of Zheng et al. and showed that a linear map   $\phi: M_{n_1\cdots n_k}\rightarrow M_{n_1\cdots n_k}$
satisfies (\ref{eq121}) for rank one matrices and nonsingular matrices has the form (\ref{eqa11}) too.

In this paper, we characterize  linear maps  $\phi: M_{n_1\cdots n_k}\rightarrow M_{n_1\cdots n_k}$ satisfying (\ref{eq121}) for only rank one matrices $A_1\otimes \cdots \otimes A_k$
with $A_i \in M_{n_i}$. In this case, the structure of maps is more complicated and the maps of the form (\ref{eqa11}) is only one of the special cases.
To state our main result, we need the following notations.
Denote by
$$\cS(\IC^m\otimes \IC^n)=\{x\otimes y: x\in \IC^m,\ y\in \IC^n\}.$$
Also $\cS\left(\IC^{n_1} \otimes \IC^{n_2} \otimes \cdots \otimes \IC^{n_k}\right)$ can be defined accordingly.
For a matrix $A=\left[ a_{ij} \right]\in M_n$,
denote by
$$\vect(A)= \left[
a_{11} \ \ a_{12} \ \cdots \ a_{1n} \ \
a_{21} \ \ a_{22} \ \cdots \ a_{2n} \ \cdots \
a_{n1} \ \ a_{n2} \ \cdots \ a_{nn}
\right]^T \in \IC^{n^2}.$$
In particular, if $A = xy^T$ is rank one matrix with $x,y\in \IC^n$,
then $\vect(xy^T) = x \otimes y$.
Given a set $S$, a partition $\{P_1,\dots,P_r\}$ of $S$ is a collection of subsets of $S$ such that
$P_i \cap P_j = \emptyset$ for $i\ne j$ and $P_1\cup \cdots \cup P_r = S$.
Here the set $P_j$ can be empty.

\medskip
We are now ready to present the main result of this paper.
\begin{theorem}\label{th0}
Let $n_1,\ldots,n_k $ be  integers larger than or equal to 2 and $m=\prod_{i=1}^kn_i$. Suppose $\phi: M_{m}\rightarrow M_{m}$ is a linear map. Then
\begin{equation}\label{rank0}
\rank(\phi(A_1\otimes \cdots \otimes A_k))=1\quad\hbox{whenever}\quad\rank (A_1\otimes \cdots \otimes A_k)=1 \quad \hbox{for all}\quad A_i \in M_{n_i},\, i =1,\dots,k
\end{equation}
if and only if
there is a partition $\{P_1,P_2,P_3,P_4\}$ of the set $K = \{1,\dots,k\}$,
an $m \times p_1p_2p_3^2$ matrix $M$ and an $m \times p_1p_2p_4^2$ matrix $N$ with $p_\ell = \prod_{i\in P_\ell} n_i$ and $p_\ell = 1$ if $P_\ell = \emptyset$, for $\ell = 1,2,3,4$, satisfying
$$\Ker(M)\, \cap \, \cS\left( \bigotimes_{i\in P_1 \cup P_2} \IC^{n_i} \otimes \bigotimes_{j\in P_3} (\IC^{n_j} \otimes \IC^{n_j}) \right) = \{0\}
\hbox{ and }
\Ker(N)\, \cap\, \cS\left( \bigotimes_{i\in P_1 \cup P_2} \IC^{n_i} \otimes \bigotimes_{j\in P_4} (\IC^{n_j} \otimes \IC^{n_j}) \right) = \{0\}$$
such that
\begin{equation}\label{eqb11}
\phi(A_1\otimes \cdots \otimes A_k) = M \left( \bigotimes_{i\in P_1} A_i
\otimes \bigotimes_{i\in P_2} A_i^T \otimes \bigotimes_{i\in P_3} \vect(A_i) \otimes \bigotimes_{i\in P_4} \vect^T(A_i) \right) N^T.
\end{equation}
Furthermore, for any given partition $\{P_1,P_2,P_3,P_4\}$ of $K$, there always exists some $M$ and $N$ that satisfy the above kernel condition, except the case $k =2$, $K = \{1,2\}$, $2\in  \{n_1,n_2\}$, and
$(P_1,P_2,P_3,P_4)  = (\emptyset, \emptyset, K, \emptyset)$
or  $(\emptyset, \emptyset, \emptyset, K)$.
Here, the notations $\bigotimes_{i\in P} \IC^{n_i}$,
$\bigotimes_{j\in P} (\IC^{n_j} \otimes \IC^{n_j})$,
and $\bigotimes_{i\in P} A_i^\dag$
vanish if $P = \emptyset$,
where $A_i^\dag = A_i, A_i^T, \vect(A_i)$, or $\vect^T(A_i)$.
\end{theorem}
Shortly after the authors obtained the above result, they learned via a private communication that, by using another structure result of Westwick \cite{RW2,RW3}, Lim \cite{Lim4} has also obtained a characterization of linear maps between rectangular matrices over an arbitrary field that is rank one non-increasing on tensor products of matrices. In the same project, Lim also considered linear maps sending tensor products of (non)-symmetric rank one
matrices to  (non)-symmetric rank one matrices.

\medskip
The rest of the paper is organized as follows.
In Section 2, the bipartite case ($k=2$) of the main result will be discussed and examples will be given to demonstrate the importance of the kernel condition for the matrices $M$ and $N$
stated in Theorem \ref{th0}.
The proof of the main result and related corollaries will be presented in Section 3.

\section{Bipartite case}
In this section, we will focus on the bipartite case (when $k = 2$).
Let $\{E_{11},\dots,E_{mm}\}$ be the standard basis of $M_m$. A matrix $X \in M_{mn}$  can be expressed as
$$X = \begin{bmatrix}
X_{11} & \cdots & X_{1m} \cr
\vdots & \ddots & \vdots \cr
X_{m1} & \cdots & X_{mm}
\end{bmatrix}
= \sum_{1\le i,j \le m} E_{ij} \otimes X_{ij}
\quad\hbox{with}\quad
X_{ij} \in M_n.$$
The partial transposes of $X$ on the first and the second system are defined by
$$X^{PT_1} = \sum_{1\le i,j \le m} E_{ji} \otimes X_{ij} \quad
\hbox{and}\quad
X^{PT_2} = \sum_{1\le i,j \le m} E_{ij} \otimes X_{ij}^T.$$
Also denote by
$$X^{R_1} =   \sum_{1\le i,j \le m} \vect(E_{ij}) \otimes X_{ij}
\quad \hbox{and}\quad
X^{R_2} =   \sum_{1\le i,j \le m} E_{ij} \otimes \vect(X_{ij} ).$$
Furthermore, define the $m^2 \times n^2$ realigned matrix of $X$ by
$$X^R = \sum_{1\le i,j \le m} \vect(E_{ij}) \otimes \vect^T(X_{ij}).$$
In particular,
$X^{PT_1} = X_1^T \otimes X_2$,
$X^{PT_2} = X_1 \otimes X_2^T$,
$X^{R_1} = \vect(X_1) \otimes X_2$,
$X^{R_2} = X_1 \otimes \vect(X_2)$,
and
$X^R = \vect(X_1) \otimes \vect^T(X_2)$
if $X = X_1 \otimes X_2$.

Finally, for any two linear maps $\psi_1$ and $\psi_2$ on matrix spaces, we say that these two maps are permutationally equivalent if there are permutation matrices $P$ and $Q$ such that $\psi_2(A) = P \psi_1(A) Q$ for all $A$. For example, it is clear that $A\mapsto \vect(A)$ and $A\mapsto \vect(A^T)$ are permutationally equivalent.

\begin{proposition}\label{prop}
Let $n_1,n_2$ be positive integers 
and $m = n_1n_2$.
Given $\psi_P:M_m \to M_m$ defined by $\psi_P(A) =  A^{PT_j}$ with $j \in \{1,2\}$.
The composite map $\psi_R\circ \psi_P$ is permutationally equivalent to the map $\psi_R$ , when $\psi_R$ is one of the following maps.
$$\hbox{(i) } A\mapsto A^{R_j},\quad
\hbox{(ii) }  A\mapsto A^{R}, \quad\hbox{or}\quad
\hbox{(iii) } A\mapsto \vect(A).$$
\end{proposition}

\begin{proof}
For $j = 1,2$, it is obvious that there is a permutation matrix $P_j \in M_{n_j}$ such that
$\vect(X_j^T) = P_j\, \vect(X_j)$ for all $X_j \in M_{n_j}$.
Also there is a permutation matrix $P_{12} \in M_m$ such that
$\vect(X_1 \otimes X_2) = P_{12} \left( \vect(X_1) \otimes \vect(X_2) \right)$
for all $X_i \in M_{n_i}$, $ i = 1,2$.
We now consider the case when $j = 1$. The case $j=2$ can be proved in a similar way.

First suppose $\psi_R:A\mapsto A^{R_1}$.
For any $X_i \in M_{n_i}$, $i=1,2$,
\begin{multline*}
\psi_R\circ\psi_P(X_1\otimes X_2)
= \left( (X_1 \otimes X_2)^{PT_1} \right)^{R_1}
= (X_1^T \otimes X_2)^{R_1}
= \vect(X_1^T) \otimes X_2  \\
= (P_1 \otimes I_{n_2}) (\vect(X_1) \otimes X_2) = (P_1 \otimes I_{n_2})(X_1 \otimes X_2)^{R_1} =  (P_1 \otimes I_{n_2}) \psi_R(X_1\otimes X_2).
\end{multline*}
By linearity of the two maps, we conclude that
$\psi_R\circ\psi_P(A) = (P_1 \otimes I_{n_2}) \psi_R(A)$ for all $A \in M_m$.

Suppose now $\psi_R:A\mapsto A^{R}$.
For any $X_i \in M_{n_i}$, $i=1,2$,
\begin{multline*}
\psi_R\circ\psi_P(X_1\otimes X_2)
= \left( (X_1 \otimes X_2)^{PT_1} \right)^{R}
= (X_1^T \otimes X_2)^{R}
= \vect(X_1^T) \otimes \vect^T(X_2)  \\
= (P_1 \otimes I_{n_2})  (\vect(X_1) \otimes \vect^T(X_2)) = (P_1 \otimes I_{n_2})  (X_1 \otimes X_2)^{R} = (P_1 \otimes I_{n_2}) \psi_R(X_1\otimes X_2).
\end{multline*}
Thus, the same conclusion holds.
Finally assume $\psi_R:A\mapsto \vect(A)$.
For any $X_i \in M_{n_i}$, $i=1,2$,
\begin{multline*}
\psi_R\circ\psi_P(X_1\otimes X_2)
= \vect(X_1^T \otimes X_2)
= P_{12} \left( \vect(X_1^T) \otimes \vect(X_2)\right)
= P_{12} (P_1 \otimes I_{n_2} ) \left( \vect(X_1) \otimes \vect(X_2)\right)  \\
 = P_{12}(P_1 \otimes I_{n_2})P_{12}^T \vect (X_1 \otimes X_2) =  P_{12}(P_1 \otimes I_{n_2})P_{12}^T \psi_R(X_1\otimes X_2).
\end{multline*}
Again by linearity of the maps, we conclude that
$\psi_R\circ\psi_P(A) = P_{12}(P_1\otimes I_{n_2}) P_{12}^T \psi_R(A)$ for all $A \in M_m$.
\end{proof}

It turns out that for the bipartite case ($k=2$), Theorem \ref{th0} can be expressed in terms of partial transpose and realigned matrix as follows.
\begin{theorem}\label{th2}
Let $n_1,n_2$ be integers larger than or equal to two and $m = n_1n_2$.
Suppose $\phi: M_{m} \to M_{m}$ is a linear map. Then
\begin{equation}\label{rank2}
\rank(\phi(A_1\otimes A_2))=1\quad\hbox{whenever}\quad \rank (A_1 \otimes A_2)=1 \quad \hbox{for all } A_i \in M_{n_i},\, i = 1,2,
\end{equation}
if and only if $\phi = \psi_T\circ \psi_M \circ \psi_R \circ \psi_P$,
where
\begin{enumerate}
\item[\rm (i)] $\psi_P: A \mapsto A$, $A\mapsto A^{PT_1}$ or $A\mapsto A^{PT_2}$;
\item[\rm (ii)] $\psi_R: A\mapsto A$, $A\mapsto A^{R_1}$, $A\mapsto A^{R_2}$, $A\mapsto A^{R}$ or $A\mapsto \vect(A)$ ;
\item[\rm (iii)] $\psi_M: A \mapsto MAN^T$;
\item[\rm (iv)] $\psi_T: A\mapsto A$ or $A\mapsto A^T$,
\end{enumerate}
which has totally $16$ different forms,
and $M$ and $N$ are matrices of appropriate sizes satisfying
\begin{enumerate}
\item[\rm (1)] $\Ker(M) \cap \cS\left( \IC^{n_1} \otimes \IC^{n_2} \right) = \{0\}$ and
$\Ker(N) \cap \cS\left( \IC^{n_1} \otimes \IC^{n_2} \right) = \{0\}$ if $\psi_R$ is the map $A\mapsto A$;
\item[\rm (2)] $\Ker(M) \cap \cS\left( \IC^{n_1} \otimes \IC^{n_1} \right) = \{0\}$ and
$\Ker(N) \cap \cS\left( \IC^{n_2} \otimes \IC^{n_2} \right) = \{0\}$ if $\psi_R$ is the map $A\mapsto A^R$;
\item[\rm (3)] $\Ker(M) \cap \cS\left( \IC^{n_1} \otimes \IC^{n_1} \otimes \IC^{n_2} \right) = \{0\}$ and
$N$ has full column rank equal to $n_2$ if $\psi_R$ is the map $A\mapsto A^{R_1}$;
\item[\rm (4)] $\Ker(M) \cap \cS\left( \IC^{n_1} \otimes \IC^{n_2} \otimes \IC^{n_2} \right) = \{0\}$ and
$N$ has full column rank equal to $n_1$ if $\psi_R$ is the map $A\mapsto A^{R_2}$;
\item[\rm (5)] $\Ker(M) \cap \cS\left( \IC^{n_1} \otimes \IC^{n_2} \otimes \IC^{n_1} \otimes \IC^{n_2} \right) = \{0\}$ and $N$ is an $m \times 1$ nonzero matrix if $2 \notin \{n_1,n_2\}$ and $\psi_R$ is the map $A\mapsto \vect(A)$.
\end{enumerate}
\end{theorem}

\begin{proof}
It is easy to verify that the two maps
$$X_1 \otimes X_2 \mapsto X_1 \otimes X_2
\quad\hbox{and}\quad
X_1 \otimes X_2 \mapsto X_2 \otimes X_1$$
are premuationally similar.
Applying Theorem \ref{th0} with $k = 2$ and taking the above observation into account,
the equation (\ref{eqb11}) can be reduced to the following 16 cases.
\begin{enumerate}[1)]
\item $(P_1,P_2,P_3,P_4) = ( \{1\}, \{2\}, \emptyset, \emptyset)$
and $\phi(A_1 \otimes A_2) = M(A_1 \otimes A_2^T) N^T = M (A_1\otimes A_2)^{PT_2}N^T$.
\item $(P_1,P_2,P_3,P_4) = ( \{2\}, \{1\}, \emptyset, \emptyset)$
and $\phi(A_1 \otimes A_2) = M(A_1^T \otimes A_2) N^T = M (A_1\otimes A_2)^{PT_1}N^T$.
\item $(P_1,P_2,P_3,P_4) = ( \{1\}, \emptyset, \{2\}, \emptyset)$
and $\phi(A_1 \otimes A_2) = M(A_1 \otimes \vect(A_2)) N^T = M (A_1\otimes A_2)^{R_2}N^T$.
\item $(P_1,P_2,P_3,P_4) = ( \{2\}, \emptyset, \{1\}, \emptyset)$
and $\phi(A_1 \otimes A_2) = M( \vect(A_1) \otimes A_2) N^T = M (A_1\otimes A_2)^{R_1}N^T$.
\item $(P_1,P_2,P_3,P_4) = ( \{1\}, \emptyset, \emptyset, \{2\})$
and $\phi(A_1 \otimes A_2) = M(A_1 \otimes \vect^T (A_2)) N^T = \left( N ( (A_1\otimes A_2)^{PT_1})^{R_2}M^T \right)^T$.
\item $(P_1,P_2,P_3,P_4) = (  \{2\}, \emptyset, \emptyset, \{1\})$
and $\phi(A_1 \otimes A_2) = M(\vect^T (A_1) \otimes A_2) N^T = \left( N ( (A_1\otimes A_2)^{PT_2})^{R_1}M^T \right)^T$.
\item $(P_1,P_2,P_3,P_4) = ( \emptyset, \{1\}, \{2\},  \emptyset)$
and $\phi(A_1 \otimes A_2) = M(A_1^T \otimes \vect (A_2)) N^T = M ((A_1\otimes A_2)^{PT_1})^{R_2}N^T$.
\item $(P_1,P_2,P_3,P_4) = (  \emptyset, \{2\}, \{1\},  \emptyset)$
and $\phi(A_1 \otimes A_2) = M( \vect(A_1) \otimes A_2^T )) N^T = M ((A_1\otimes A_2)^{PT_2})^{R_1}N^T$.
\item $(P_1,P_2,P_3,P_4) = (  \emptyset, \{1\}, \emptyset, \{2\})$
and $\phi(A_1 \otimes A_2) = M(A_1^T \otimes \vect^T (A_2)) N^T = \left(N (A_1\otimes A_2)^{R_2}M^T\right)^T$.
\item $(P_1,P_2,P_3,P_4) = (  \emptyset, \{2\}, \emptyset, \{1\})$
and $\phi(A_1 \otimes A_2) = M( \vect^T (A_1) \otimes A_2^T )) N^T = \left(N (A_1\otimes A_2)^{R_1}M^T\right)^T$.
\item $(P_1,P_2,P_3,P_4) = (  \emptyset, \emptyset, \{1\}, \{2\})$
and $\phi(A_1 \otimes A_2) = M(\vect(A_1) \otimes \vect^T (A_2)) N^T = M (A_1\otimes A_2)^{R}N^T$.
\item $(P_1,P_2,P_3,P_4) = (  \emptyset, \emptyset, \{2\}, \{1\})$
and $\phi(A_1 \otimes A_2) = M(\vect^T(A_1) \otimes \vect (A_2)) N^T = \left( N (A_1\otimes A_2)^{R}M^T\right)^T$.
\item $(P_1,P_2,P_3,P_4) = (  \{1,2\}, \emptyset, \emptyset, \emptyset \}$
and $\phi(A_1 \otimes A_2) = M(A_1 \otimes A_2) N^T$.
\item $(P_1,P_2,P_3,P_4) = ( \emptyset, \{1,2\}, \emptyset, \emptyset)$
and $\phi(A_1 \otimes A_2) = M(A_1^T \otimes A_2^T) N^T = \left( N (A_1\otimes A_2)M^T \right)^T$.
\item $(P_1,P_2,P_3,P_4) = ( \emptyset,\emptyset,  \{1,2\}, \emptyset)$ and 
$\phi(A_1 \otimes A_2) = M\left( \vect (A_1)  \otimes \vect( A_2)  \right) N^T$ \\[1mm] \hspace*{74mm} $=  M P_{12}^T \left( \vect(A_1 \otimes A_2) \right) N^T$.
\item $(P_1,P_2,P_3,P_4) = (  \emptyset,\emptyset,\emptyset,  \{1,2\}) $ 
and $\phi(A_1 \otimes A_2) = M\left( \vect^T (A_1)  \otimes \vect^T( A_2)  \right) N^T  $ \\[1mm] \hspace*{74mm} $= \left(N P_{12}^T (\vect \left(A_1 \otimes A_2) \right) M^T \right)^T$.
\end{enumerate}
Here, $M$ and $N$ are matrices with appropriate size,
and satisfy the kernel condition in Theorem \ref{th0}
(In some cases, the roles of $M$ and $N$ may interchange).
Also the cases 15) and 16) hold only when $2\notin \{n_1,n_2\}$.
In all these cases, the map $\phi$ can be represented by
$A\mapsto \psi_T \circ \psi_M \circ \psi_R \circ \psi_P(A)$
where $\psi_P$, $\psi_M$, $\psi_R$, $\psi_T$ are of the forms in (i), (ii), (iii) and (iv) respectively.
In particular for case 15), we have $\Ker(M) \cap \cS\left( \IC^{n_1} \otimes \IC^{n_1} \otimes \IC^{n_2} \otimes \IC^{n_2} \right) = \emptyset$
by Theorem \ref{th0} and hence
$\Ker(MP_{12}^T) \cap \cS\left( \IC^{n_1} \otimes \IC^{n_2} \otimes \IC^{n_1} \otimes \IC^{n_2} \right) = \emptyset$.
Similar for the case 16).
Furthermore, by Proposition \ref{prop}, if $\psi_P$ is a partial transport map with respect to the $j$th subsystem, $\psi_R\circ \psi_P$ is permutationally equivalent to $\psi_R$, when $\psi_R$ has the form
$A\mapsto A^{R_j}$, $A\mapsto A^R$ or $A\mapsto \vect(A)$.
Therefore, instead of $15$ different types, there are actually only $9$ different types of compositions of $\psi_R \circ \psi_P$.
Finally, since $(A^{PT_1})^T = A^{PT_2}$ and $(A^{PT_2})^T = A^{PT_1}$,
the maps $A\mapsto (MA^{PT_1}N^T)^T$
and $A\mapsto (MA^{PT_2}N^T)^T$
are the same as
$A\mapsto N^T A^{PT_2} M$ and
$A\mapsto N^T A^{PT_1} M$, respectively.
Therefore, the map $\psi_T\circ \psi_M \circ \psi_R \circ \psi_P$ has totally $16$ different forms only.
\end{proof}

In the following, we give some low dimensional examples of $M$ and $N$ that satisfy the conditions (2), (3) and (5) of Theorem \ref{th2}. 

\begin{example} \rm
Assume $(n_1,n_2) = (2,3)$ and define the $6 \times 4$ matrix $M$ and the $6 \times 9$ matrix $N$ by
$$M = \begin{bmatrix}
1 & 0 & 0 & 0 \cr
0 & 1 & 0 & 0 \cr
0 & 0 & 1 & 0 \cr
0 & 0 & 0 & 1 \cr
0 & 0 & 0 & 0  \cr
0 & 0 & 0 & 0
\end{bmatrix}
\quad\hbox{and}\quad
N = \begin{bmatrix}
1 & 0 & 0 & 0 & -1 & 0 & 0 & 0 & 0 \cr
0 & 1 & 0 & 0 & 0 & -1 & 0 & 0 & 0 \cr
0 & 0 & 1 & 0 & 0 & 0 & -1 & 0 & 0 \cr
0 & 0 & 0 & 1 & 0 & 0 & 0 & -1 & 0 \cr
0 & 0 & 0 & 0 & 0 & 0 & 0 &  0 & -1 \cr
0 & 0 & 0 & 0 & 0 & 0 & 0 & 0 & 0
\end{bmatrix}.$$
Clearly, $\rank(M) = 4$ and $\rank(N) = 5$. Also
$$\Ker(M) = \{0\} \quad\hbox{and}\quad \Ker(N) = \left\{
\begin{bmatrix} a & b & c & d & a & b & c  & d & 0\end{bmatrix}^T: a,b,c,d\in \IC \right\}.$$
Therefore, $\Ker(N)$  does not contain any nonzero element in $\cS(\IC^3 \otimes \IC^3)$. Then the map $A\mapsto MA^RN^T$
satisfies the condition (\ref{rank2}) and its range space contains matrices of rank at most $4$ only.
\end{example}

\begin{example} \rm
Assume $(n_1,n_2) = (2,3)$ and define the $6 \times 12$ matrix $M$ and the $6 \times 3$ matrix $N$ by
$$M = \begin{bmatrix} I_6 & \hat M \end{bmatrix}
= \left[\begin{array}{cccccccccccc}
1 & 0  & 0 & 0 & 0 & 0 & 0  & 1 & 0 & 0  & 0 & 0  \cr
0 & 1  & 0 & 0 & 0 & 0 & 0 & 0 & 1 & 0 & 0 & 0 \cr
0 & 0  & 1 & 0 & 0 & 0 & 0  & 0 & 0 & 1 & 0 & 0 \cr
0 & 0  & 0 & 1 & 0 & 0 & 1 & 0 & 0 & 0 & 1 & 0 \cr
0 & 0  & 0 & 0 & 1 & 0 & 0 & 0 & 0 & 0 & 0 & 1 \cr
0 & 0  & 0 & 0 & 0 & 1 & 0 & 0 & -1 & 0 & 0 & 0
\end{array}\right]
\quad\hbox{and}\quad
N = \begin{bmatrix}
1 & 0 & 0 \cr
0 & 1 & 0 \cr
0 & 0 & 1 \cr
0 & 0 & 0 \cr
0 & 0 & 0 \cr
0 & 0 & 0
\end{bmatrix}.$$
Clearly, $\Ker(N)  = \{0\}$. Suppose $M(x\otimes y \otimes z) = 0$ for some nonzero $x,y\in \IC^2$ and $z \in \IC^3$.
Then
$$0 = M(x\otimes y \otimes z) = M(x \otimes I_6) (y \otimes z) =  (x_1 I_6 + x_2 \hat M) (y \otimes z),$$
where $x = \left[ x_1 \ x_2 \right]^T$. So $(x_1 I_6 + x_2 \hat M)$ is singular and hence $x_1 = 0$ as $\det(x_1 I_6 + x_2 \hat M) = x_1^6$.
Thus, the vector $y \otimes z$ is in the kernel of $\hat M$. However, $\Ker(\hat M) = \left\{\left[a \ 0 \ 0 \ 0 \ a \ 0\right]^T: a\in \IC \right\}$,
which does not contain any nonzero element of $\cS(\IC^2 \otimes \IC^3)$.
Therefore, even $\Ker(M)$ is a $6$ dimensional subspace of $\IC^{12}$,
$\Ker(M)$ does not contain any nonzero element of $\cS(\IC^2 \otimes \IC^2 \otimes \IC^3)$.
\end{example}

\begin{example} \rm
Assume $(n_1,n_2) = (3,3)$ and define the $9 \times 81$ matrix $M$ by
$$M = \begin{bmatrix}
I_9 & R & R^2 & R^3 & -I_9 & -R & -R^2 & -R^3 & R^4
\end{bmatrix}
\quad\hbox{with}\quad
R = \begin{bmatrix}
0 & 1 & 0 & 0 & 0 & 0 & 0 & 0 & 0 \cr
0 & 0 & 1 & 0 & 0 & 0 & 0 & 0 & 0 \cr
0 & 0 & 0 & 1 & 0 & 0 & 0 & 0 & 0 \cr
-1 & 0 & 0 & 0 & 1 & 0 & 0 & 0 & 0 \cr
0 & 0 & 0 & 0 & 0 & 1 & 0 & 0 & 0 \cr
0 & 0 & 0 & 0 & 0 & 0 & 1 & 0 & 0 \cr
0 & 0 & 0 & 0 & 0 & 0 & 0 & 1 & 0 \cr
-1 & 0 & 0 & 0 & 1 & 0 & 0 & 0 & 1 \cr
0 & 0 & 0 & 0 & 0 & 0 & 0 & 0 & 0
\end{bmatrix}.
$$
Notice that
$$\Ker(R) \subseteq \Ker(R^2) \subseteq \Ker(R^3) \subseteq \Ker(R^4)
= \left\{
\begin{bmatrix} a & b & c & d & a & b & c  & d & 0\end{bmatrix}^T: a,b,c,d\in \IC \right\}.$$
Suppose $M(x \otimes y \otimes z \otimes w) = 0$ for some nonzero
$x,y,z,w \in \IC^3$.
Set $x \otimes y = \left[u_1 \ \cdots \ u_9\right]^T \in \IC^9$
and define
$$U = M (x \otimes y \otimes I_9) = (u_1-u_5) I_9 + (u_2-u_6) R
+(u_3-u_7) R^2 + (u_4-u_8) R^3 + u_9 R^4.$$
Then
$$0 = M(x \otimes y \otimes z \otimes w) = M (x\otimes y \otimes I_9) (z \otimes w)
= U (z \otimes w).$$
Now let
$$
U_5 = u_9 I_9\quad\hbox{and}\quad
U_k = (u_k - u_{k+4}) I_9 + U_{k+1} R \quad\hbox{for }k = 1,2,3,4.$$
Then it can be verified that
$$U_1 = (u_1-u_5) I_9 + \left( (u_2-u_6) I_9 + \left(  (u_3-u_7) I_9 + \left(  (u_4-u_8) I_9 + \left( u_9 I_9 \right) R \right) R \right) R \right) R = U.$$
For $k = 1,2,3,4$,
because $R$ is singular,
$U_k$ is singular if and only if $u_k - u_{k+4} = 0$, or equivalently, $U_k= U_{k+1} R$.
Furthermore, when $U_k$ is singular,
$$\Ker(U_k R^{k-1}) = \Ker(U_{k+1} R R^{k-1}) = \Ker(U_{k+1} R^{k}).$$
Suppose at least one of $U_1,\dots,U_5$ is nonsingular, say $U_\ell$ is nonsingular for some $1\le \ell \le 5$
and $U_1, \dots, U_{\ell-1}$ are all singular.
Then
$$\Ker(U) = \Ker(U_1) \subseteq \Ker(U_2 R) \subseteq \Ker(U_3 R^2) \subseteq \cdots \subseteq \Ker(U_\ell R^{\ell-1}) = \Ker (R^{\ell-1}) \subseteq \Ker(R^4).$$
But this is impossible since $U(z\otimes w) = 0$ while $\Ker(R^4)$ does not contain any nonzero element of $\cS(\IC^3\otimes \IC^3)$. Therefore, all $U_1,\dots,U_5$ are singular.
In this case, we have $u_k - u_{k+4} = 0$ for $k=1,2,3,4$ and $u_9 = 0$, or equivalently,
$x\otimes y$ has the form
$\begin{bmatrix} u_1 & u_2 & u_3 & u_4 & u_ 1& u_2 & u_3  & u_4 & 0\end{bmatrix}^T$,
and contradiction again arrived.
Thus, one can conclude that $\Ker(M)$ does not contain any nonzero element of $\cS(\IC^3 \otimes \IC^3 \otimes \IC^3 \otimes \IC^3)$.
Now take any $9 \times 1$ nonzero matrix $N$.
Then the composition map $\phi: A \mapsto M \vect(A) N^T$ satisfies condition (\ref{rank2}).
In this case,  $\rank(\phi(A)) \le 1$ for all $A\in M_9$.

\end{example}

\begin{remark} \rm
For condition (1) of Theorem \ref{th2}, both $M$ and $N$ have size $m \times m$.
In this case, any nonsingular matrices $M, N\in M_{m}$ satisfy case (1). But there exist singular matrices that satisfy the condition (1) too.
For example, when $(n_1,n_2) = (2,2)$ one can construct a rank three $4 \times 4$ matrix $M$
with $\Ker(M) = \left\{\left[a \ 0 \ 0 \ a \right]^T: a\in \IC \right\}$,
which does not contain any nonzero vector in $\cS(\IC^2 \otimes \IC^2)$.

For condition (2) of Theorem \ref{th2}, the same observation as above follows if $n_1 = n_2$.
If $n_1 < n_2$, $M$ can be chosen to be any $m \times n_1^2$ matrix with full column rank, i.e., $\rank(M) = n_1^2$.
Similarly, $N$ can be chosen to be any $m \times n_2^2$ matrix with full column rank
if $n_1 > n_2$.

\end{remark}

Finally, it has to point out that the partial transpose and realignment are two useful concept in the study of separable problem,
which is one of the most important problems in quantum information science.
Although it have been showed that the general characterization of separable states is NP-hard \cite{Gur},
researchers are interested in finding effective criterion to determine separability of a quantum state.
A quantum state (density matrix) $X$ is PPT (positive partial transpose)
if $X^{PT_1}$ (or equivalently $X^{PT_2}$) is positive semi-definite.
One of the classical and popular criteria is PPT criterion introduced by Peres \cite{Per}.
The PPT criterion states that if $X$ is separable, then $X$ is PPT
and these two conditions are equivalent if $m = n_1n_2 \le 6$ \cite{HHH}.
Another strong criterion is CCNR criterion \cite{CW, Rud}, which confirmed that
$\|X^R\|_1 \le 1$ if $X$ is separable.
It has to note that researchers also studied preservers on separable states, see
\cite{AS, FLPS, John}. In particular, the authors in \cite{FLPS} studied linear maps
that send the set of separable states onto itself in multipartite system.

\section{Proof of the main results}

In this section, we will present the proof of Theorem \ref{th0}.
The proof relies on the structure result of  Westwick \cite[Theorem 3.4]{RW} on preservers of nonzero decomposable tensors,
and we restate this result as follows.

\begin{theorem}\label{West}
Let $U_1,\dots, U_p$ and $W_1,\dots,W_q$ be finite dimensional vector spaces over a field $F$ with $\dim(U_i) \ge 2$ and define $U  = \bigotimes_{i=1}^p U_i$ and $W = \bigotimes_{j = 1}^q W_j$ to be the tensor product spaces of $U_i$ and $W_j$.
Suppose $f: U \to W$ is a linear map sending nonzero
decomposable tensors into nonzero decomposable tensors.
Then there is a partition $\{S_1,\dots,S_q\}$ of $\{1,\dots,p\}$ ($S_j$ can be an empty set)
and linear functions $f_j: \bigotimes_{i\in S_j} U_i\to W_j$ sending nonzero decomposable tensors
to nonzero vectors, such that
$$f(x_1\otimes \cdots \otimes x_p) = \bigotimes_{j=1}^q f_j \left(\otimes_{i\in S_j} x_i\right).$$
Here, $f_j$ is defined to be a nonzero constant function, i.e., $f_j(\cdot) = w_j$
for some nonzero $w_j \in W_j$, if $S_j  = \emptyset$.
\end{theorem}

We will prove the following equivalent version of Theorem \ref{th0}.

\begin{theorem}\label{th1}
Let $n_1,\ldots,n_k $ be integers larger than or equal to 2 and let $m=\prod_{i=1}^kn_i$. Suppose $\phi: M_{m}\rightarrow M_{m}$ is a linear map. Then
\begin{equation}\label{rank}
\rank(\phi(A_1\otimes \cdots \otimes A_k))=1\quad\hbox{whenever}\quad\rank (A_1\otimes \cdots \otimes A_k)=1 \quad \hbox{for all } A_i \in M_{n_i}\, i = 1,\dots,k
\end{equation}
if and only if
there are two subsets $K_1, K_2$ of $K = \{1,\dots,k\}$,
an $m \times m_1m_2$ matrix $M$ and an $m \times m^2/(m_1m_2)$ matrix $N$ with
$m_t = \prod_{i\in K_t} n_i$ or  $m_t = 1$ if $K_t = \emptyset$, $t = 1,2$, satisfying
\begin{equation}\label{eqc1}
\Ker(M) \cap \cS\left( \bigotimes_{i\in K_1} \IC^{n_i} \otimes \bigotimes_{j\in K_2} \IC^{n_j} \right) = \{0\}
\quad\hbox{and}\quad
\Ker(N) \cap \cS\left( \bigotimes_{i\notin K_1} \IC^{n_i} \otimes \bigotimes_{j\notin K_2}
\IC^{n_j}  \right) = \{0\}
\end{equation}
such that
\begin{equation}\label{eq11}
\phi\left( x_1y_1^T\otimes \cdots \otimes x_ky_k^T \right) = M \left( \bigotimes_{i\in K_1} x_i
\otimes \bigotimes_{j\in K_2} y_j \right) \left(\bigotimes_{i\notin K_1} x_i \otimes \bigotimes_{j\notin K_2} y_j \right)^T N^T\quad\hbox{for all}\quad x_i,y_i \in \IC^{n_i}.
\end{equation}
Furthermore, for any given subsets $K_1,K_2$ of $K$, there always exists some $M$ and $N$ that satisfy the above kernel condition, except the case $k =2$, $K = \{1,2\}$, $2\in  \{n_1,n_2\}$, and either $K_1 = K_2 = K$ or $K_1 = K_2 = \emptyset$.
\end{theorem}

\begin{proof}
The necessary part is clear. For the sufficient part, define a linear map $f: \IC^{m^2} \to \IC^{m^2}$ such that
$$f\left(\bigotimes_{i=1}^k (x_i \otimes y_i) \right) = \vect\left( \phi\left( \bigotimes_{i=1}^k x_iy_i^T \right) \right)
\quad\hbox{for all}\quad x_i,y_i \in \IC^{n_i},$$
and by linearity, extend the definition of $f$ to all vectors in $\IC^{m^2}$.
Recall that $\vect(A) = x \otimes y$ if $A = xy^T$ is rank one.
As $\phi$ satisfies (\ref{rank}), the map $f$ will send all
nonzero vectors of the form $\bigotimes_{i=1}^k (x_i \otimes y_i)$ to  some nonzero vectors of the form $u \otimes v \in \IC^m \otimes \IC^m$,
i.e., $f$ sends nonzero decomposable elements of $\bigotimes_{i=1}^k \IC^{n_i} \otimes \IC^{n_i}$ to
nonzero decomposable elements of $\IC^m \otimes \IC^m$.
Applying Theorem \ref{West} (\cite[Theorem 3.4]{RW}) with $p = 2k$ and $q = 2$,
there are two partitions $\{K_1, \overline{K_1}\}$ and $\{K_2, \overline{K_2}\}$ of $K= \{1,\dots,k\}$,
and linear maps $f_1: \IC^{m_1m_2} \to \IC^m$ and $f_2: \IC^{m^2/(m_1m_2)} \to \IC^m$,
where $m_t$ is defined as in statement of the theorem,
such that
$$f\left(\bigotimes_{i=1}^k (x_i \otimes y_i) \right)  = f_1 \left( \bigotimes_{i\in K_1} x_i \otimes \bigotimes_{j\in K_2} y_j \right) \otimes
f_2 \left( \bigotimes_{i \in \overline{K_1}} x_i \otimes \bigotimes_{j\in \overline{K_2}} y_j \right).$$
As $f_1$ and $f_2$ are linear, there exist an $m \times m_1m_2$ matrix $M$ and an $m \times m^2/(m_1m_2)$ matrix $N$ such that $f_1(z) = Mz$ and $f_2(w) = Nw$.
Thus, $\phi$ has the form as described in (\ref{eq11}).
Further, $f_1(z) \ne 0$ for all $z \in \bigotimes_{i\in K_1} \IC^{n_i} \otimes \bigotimes_{j\in K_2} \IC^{n_j}$
and $f_2(w) \ne 0$ for all $w \in \bigotimes_{i\notin {K_1}} \IC^{n_i} \otimes \bigotimes_{j\notin {K_2}} \IC^{n_j}$ as $\overline{K_j} = K \setminus K_j$,
and hence,
$M$ and $N$ satisfy the condition (\ref{eqc1}).
The last statement will be confirmed by Proposition \ref{prop2}.
\end{proof}

Now the equivalence of Theorems \ref{th0} and \ref{th1} can be seen as follows.

\noindent{\it Proof of Theorem \ref{th0}.}
Suppose $\phi$ satisfies the rank condition (\ref{rank0}). Then Theorem \ref{th1} implies that $\phi$ has the form
(\ref{eq11}) with $M$ and $N$ satisfying (\ref{eqc1}).
Set
$P_1 = K_1 \backslash K_2$, $P_2 = K_2 \backslash K_1$, $P_3 = K_1 \cap K_2$, and $P_4 = K \backslash (K_1 \cup K_2)$.
First, there exists a permutation matrix $Q_x$ such that
for any $x_i,y_i \in \IC^{n_i}$,
$$
Q_x \left( \bigotimes_{i\in P_1} x_i  \otimes \bigotimes_{j\in P_2} y_j
\otimes \bigotimes_{k\in P_3} (x_k \otimes y_k) \right)
= \left( \bigotimes_{i\in P_1} x_i \otimes \bigotimes_{i\in P_3} x_i \otimes \bigotimes_{j\in P_2} y_j
\otimes \bigotimes_{j\in P_3} y_j \right)
=  \left( \bigotimes_{i\in K_1} x_i \otimes \bigotimes_{j\in K_2} y_j \right).
$$
Similarly, there exists another permutation matrix $Q_y$ such that
for any $x_i,y_i \in \IC^{n_i}$,
$$Q_y \left(\bigotimes_{j\in P_1} y_j \otimes \bigotimes_{i\in P_2} x_i \otimes
 \bigotimes_{k\in P_4} (x_k \otimes y_k) \right)
 = \left(\bigotimes_{i\in P_2} x_i \otimes \bigotimes_{i\in P_4} x_i \otimes
\bigotimes_{j\in P_1} y_j \otimes \bigotimes_{j\in P_4} y_j \right)
= \left(\bigotimes_{i\notin K_1} x_i \otimes \bigotimes_{j\notin K_2} y_j \right).
$$
Now for any rank one matrix $A_i = x_iy_i^T$ with $x_i,y_i \in \IC^{n_i}$, $i = 1,\dots, k$,
\begin{eqnarray*}
&&\phi\left( A_1\otimes \cdots \otimes A_k \right) 
= \phi\left( x_1y_1^T\otimes \cdots \otimes x_ky_k^T \right)  \\[1mm]
&=& M \left( \bigotimes_{i\in K_1} x_i \otimes \bigotimes_{j\in K_2} y_j \right)
\left(\bigotimes_{i\notin K_1} x_i \otimes \bigotimes_{j\notin K_2} y_j \right)^T N^T \\
&=& M Q_x \left( \bigotimes_{i\in P_1} x_i  \otimes \bigotimes_{j\in P_2} y_j
\otimes \bigotimes_{k\in P_3} (x_k \otimes y_k) \right)
\left(\bigotimes_{j\in P_1} y_j \otimes \bigotimes_{i\in P_2} x_i \otimes
 \bigotimes_{k\in P_4} (x_k \otimes y_k) \right)^TQ_y^T  N^T \\[1mm]
&=& M Q_x \left( \bigotimes_{i\in P_1} x_i  \otimes \bigotimes_{j\in P_2} y_j
\otimes \bigotimes_{k\in P_3} (x_k \otimes y_k) \right)
\left(\bigotimes_{j\in P_1} y_j^T \otimes \bigotimes_{i\in P_2} x_i^T \otimes
 \bigotimes_{k\in P_4} (x_k \otimes y_k)^T \right) Q_y^T  N^T \\
&=& M Q_x \left( \left( \bigotimes_{i\in P_1} x_i  \right) \left(  \bigotimes_{i\in P_1} y_i^T \right)
\otimes    \left( \bigotimes_{j\in P_2} y_j  \right) \left(  \bigotimes_{j\in P_2} x_j^T \right)
\otimes \bigotimes_{k\in P_3} (x_k \otimes y_k)
\otimes \bigotimes_{k\in P_4} (x_k \otimes y_k)^T \right) Q_y^T  N^T \\
&=& M Q_x \left( \bigotimes_{i\in P_1} x_i  y_i^T \otimes \bigotimes_{j\in P_2} y_j x_j^T
\otimes \bigotimes_{k\in P_3} (x_k \otimes y_k)
\otimes \bigotimes_{k\in P_4} (x_k \otimes y_k)^T \right) Q_y^T  N^T \\
&=& M Q_x \left( \bigotimes_{i\in P_1} x_i  y_i^T \otimes \bigotimes_{j\in P_2} (x_j y_j^T)^T
\otimes \bigotimes_{k\in P_3} \vect(x_k y_k^T)
\otimes \bigotimes_{k\in P_4} \vect^T (x_k y_k^T)\right) Q_y^T  N^T \\
&=& M Q_x \left( \bigotimes_{i\in P_1} A_i \otimes \bigotimes_{j\in P_2} (A_j)^T
\otimes \bigotimes_{k\in P_3} \vect(A_k)
\otimes \bigotimes_{k\in P_4} \vect^T (A_k)\right) Q_y^T  N^T.
\end{eqnarray*}
By linearity, the equality holds for any matrix $A_i \in M_{n_i}$ and hence we have (\ref{eqb11}).
Finally, the kernel condition can be easily reduced from (\ref{eqc1}).
\qed

\medskip
Next we show that the matrices $M$ and $N$ in Theorem \ref{th1} (equivalently, Theorem \ref{th0})
always exist, except for two special cases, namely, when $k = 2$, $K = \{1,2\}$, $2 \in \{n_1,n_2\}$,
and $K_1 = K_2 = K$ or $K_1 =K_2 = \emptyset$.
For simplicity, we focus on the existence of $M$.
For positive integers $p_1,\dots,p_r$,
denote by $\mathcal{E}(p_1,\ldots,p_r)$ the collection of subspaces $\cV$ of
$\IC^{p_1\cdots p_r}$ such that
$$\cV \cap \cS\left( \IC^{p_1}\otimes \cdots\otimes\IC^{p_r} \right)=\{0\}.$$
The subspace $\cV$ is called a completely entangled subspace in \cite{KP}.
In the same paper, the author also obtained
the maximum dimension of $\cV$ in $\cE(p_1,\dots,p_r)$ as follows.


\begin{proposition}{\rm \cite[Theorem 1.5]{KP}}\label{le1}
Let $p_1,\ldots,p_r$ be positive integers. Then
$$\max_{\cV \in\cE(p_1,\ldots,p_r)}\dim \cV =\prod_{i=1}^r p_i-\sum_{i=1}^r p_i+r-1.$$
\end{proposition}

It has to mention that an explicit construction for maximum completely entangled subspace for bipartite case ($r = 2$) was also given in \cite{KP}.
Based on the above proposition, we can deduce the following result which showed that the matrix $M$ always exists, except for one special case.

 \begin{proposition}\label{prop2}
Let $n_1,\ldots,n_k $ be integers larger than or equal to 2, $K=\{1,\ldots,k\}$, and $K_1,K_2 \subseteq K$.
Define $m = \prod_{i \in K} n_i$ and $m_t = \prod_{i\in K_t} n_i$ for $t = 1,2$.
Then there always exists an $m \times m_1m_2$ matrix $M$ such that
$$\Ker(M) \cap \cS\left( \bigotimes_{i\in K_1} \IC^{n_i} \otimes \bigotimes_{j\in K_2} \IC^{n_j} \right) = \{0\},$$
except the case when $K_1 = K_2 = K = \{1,2\}$ and $2 \in \{n_1,n_2\}$.
\end{proposition}

\begin{proof}
If $m \ge m_1 m_2$, then any $m \times m_1m_2$ matrix with full column rank, i.e., $\rank(M) = m_1m_2$ will satisfy
the kernel condition. Let us assume that $m < m_1m_2$.
Notice that $\Ker(M)$ is a subspace of $\IC^{m_1m_2}$.
By Proposition \ref{le1}, the maximum dimension of subspace of $\IC^{m_1m_2}$ which
does not contain any nonzero element of $\cS\left(\bigotimes_{i\in K_1} \IC^{n_i} \otimes \bigotimes_{j\in K_2} \IC^{n_j} \right)$ is equal to
$$d(K_1,K_2) := m_1m_2 - \sum_{i \in K_1} n_i - \sum_{j\in K_2} n_j + |K_1| + |K_2| - 1
= m_1m_2 - \sum_{i\in K_1} (n_i-1) - \sum_{j \in K_2} (n_j -1) -1.$$
On the other hand, $\dim \Ker(M) \ge m_1m_2 -m$ for all $m \times m_1m_2$ matrices and the equal
holds when $M$ has full row rank, i.e., $\rank(M) = m$.
Therefore, the $m \times m_1m_2$ matrix $M$ satisfying the kernel condition will always exist when
$d(K_1,K_2) \ge m_1m_2 - m$, or equivalently,
\begin{equation}\label{ineq}
m \ge  \sum_{i\in K_1} (n_i-1) + \sum_{j \in K_2} (n_j -1)  + 1.
\end{equation}
Notice that for any positive integers $a_1,\dots,a_k$,
$$\prod_{j=1}^k (a_j+1) \ge \sum_{1\le i < j \le k} a_ia_j +
\sum_{j=1}^k a_j  + 1
\ge \sum_{j=1}^k a_j + \sum_{j=1}^k a_j + 1
= 2 \sum_{j=1}^k a_j  + 1 \quad \hbox{if} \quad k \ge 3.$$
Assume $k \ge 3$ and take $a_j = n_j -1$ in the above equation, we have
$$m = \prod_{i\in K} n_i \ge 2 \sum_{i\in K} (n_i-1) + 1
\ge \sum_{i\in K_1} (n_i - 1) + \sum_{j \in K_2} (n_j - 1) + 1.$$
Therefore, the matrix $M$ exists when $k \ge 3$.
For $k = 2$,
$$m = \prod_{j=1}^2 n_j = 2\sum_{j=1}^2 (n_j-1) + \prod_{j=1}^2 (n_j-2)
\ge \sum_{i\in K_1} (n_i-1) + \sum_{j\in K_2} (n_j -1) + 0,$$
and the equality holds if and only if $K_1 = K_2 = K = \{1,2\}$ and at least one of $n_i$ is equal to $2$.
In all other cases, the above inequality is strict, and therefore, the inequality (\ref{ineq}) holds.
Finally, suppose $K_1 = K_2 = K = \{1,2\}$ and $2\in \{n_1,n_2\}$. We may assume $n_1 = 2$,
then
$$d(K_1,K_2) = 4n_2^2 - 2n_2 - 1 < 4n_2^2 - 2n_2 \le \dim \Ker(M) \quad \hbox{for any $(2n_2) \times (2n_2)(2n_2)$
matrix $M$}.$$
Therefore, there is no matrix $M$ satisfying the kernel condition in this case.
\end{proof}
After we obtained the above result, it has come to our attention that Lim \cite{Lim3} has already given a necessary and sufficient condition for the existence of linear maps preserving nonzero decomposable tensor for any algebraically close field, see \cite[Proposition 2.8]{Lim3}. This existence condition is actually equivalent to the inequality (\ref{ineq}) in our proof. Also a similar conclusion on linear maps on matrix space is obtained in a recent work of Lim in \cite{Lim4} too.

\medskip
Finally, we apply Theorem \ref{th0} to obtain the following corollaries, which generalize the results of Zheng et al. \cite{ZXF} and Lim \cite{LIM}.

\begin{corollary}\label{co1}
Let $n_1,\ldots,n_k $ be integers larger than or equal to 2 and  let $m=\prod_{i=1}^kn_i$. Suppose $\phi: M_{m}\rightarrow M_{m}$ is a linear map. If   \begin{equation*}
\rank(\phi(A_1\otimes \cdots \otimes A_k))=1\quad\hbox{whenever}\quad\rank (A_1\otimes \cdots \otimes A_k)=1 \quad \hbox{for all } A_i \in M_{n_i},\, i = 1,\dots,k,
\end{equation*}
and there is a matrix $X_1\otimes \cdots \otimes X_k$ with $X_i\in M_{n_i}$ and $\rank X_i>1$ for $i=1,\ldots,k$ such that
$$\rank(\phi(X_1\otimes \cdots \otimes X_k))=\rank (X_1\otimes \cdots \otimes X_k),$$
  then $\phi$ has the form
 \begin{equation*}
 \phi(A_1\otimes \cdots \otimes A_k)=M(\psi_1(A_1)\otimes \cdots\otimes \psi_k(A_k))N^T
 \end{equation*}
 for all $A_i\in M_{n_i}$ with $i=1,\ldots,k$, where $\psi_j$ is the identity map or the transpose map for $j=1,\ldots,k$, and $M, N\in M_m$ satisfy
$$\Ker(M) \cap \cS\left(\IC^{n_1} \otimes \cdots \otimes \IC^{n_k}\right) = \{0\}
\quad\hbox{and}\quad
\Ker(N) \cap \cS\left(\IC^{n_1} \otimes \cdots \otimes \IC^{n_k}\right) = \{0\}.$$
\end{corollary}

\begin{proof}
By Theorem \ref{th0}, $\phi$ has the form (\ref{eqb11}) with partition $\{P_1,P_2,P_3,P_4\}$
as defined in the theorem. Notice that
$\rank(\vect(A)) = 1$ for any matrix $A$.
Suppose $P_3 \cup P_4 \ne \emptyset$. Then
\begin{eqnarray*}
\rank\left(\phi(X_1\otimes \cdots \otimes X_k) \right)
&\le&
\left( \prod_{j \in P_1\cup P_2} \rank(X_j) \right)
\left( \prod_{j \in P_3 \cup P_4} \rank(\vect(X_j)) \right) \\[1mm]
&=&
\left( \prod_{j \in P_1 \cup P_2} \rank(X_j) \right)
< \rank\left(X_1\otimes \cdots \otimes X_k\right),
\end{eqnarray*}
which contradicts the assumption. So $P_3 \cup P_4 = \emptyset$ and $\phi$ has the asserted from.
\end{proof}

\begin{corollary}\label{co2}
Let $n_1,\ldots,n_k $ be integers larger than or equal to 2 and  let $m=\prod_{i=1}^kn_i$. Suppose $\phi: M_{m}\rightarrow M_{m}$ is a linear map.
Then
\begin{equation*}
\rank(\phi(A_1\otimes \cdots \otimes A_k))=1\quad\hbox{whenever}\quad\rank (A_1\otimes \cdots \otimes A_k)=1 \quad \hbox{for all } A_i \in M_{n_i},\, i = 1,\dots,k
\end{equation*}
and $\phi(X_1\otimes \cdots \otimes X_k)$ is nonsingular for some
$X_1\otimes \cdots \otimes X_k$ with $X_i\in M_{n_i}$
if and only if
there exist nonsingular matrices $M, N\in M_m$ such that
 \begin{equation*}
 \phi(A_1\otimes \cdots \otimes A_k)=M\left(\psi_1(A_1)\otimes \cdots\otimes \psi_k(A_k)\right)N
\quad\hbox{ for all}\quad A_i\in M_{n_i},\, i=1,\ldots,k,
 \end{equation*}
 where $\psi_j$, $j=1,\dots,k$ is either the identity map or the transpose map.
\end{corollary}

\begin{proof}
The sufficient part is clear. For the necessary part, by Theorem \ref{th1} and a similar argument as in the proof of Corollary \ref{co1}, one can show that
$P_3 \cup P_4 = \emptyset$ and $M$ and $N$ are both nonsingular. Then the result follows.
\end{proof}

\section*{Acknowledgement}
The authors would like to thank the referee for his/her valuable comments which helped to improve this manuscript.
They also thank Professor Ming Huat Lim for his useful comments, drawing their attention to Ref \cite{Lim3} and sharing them his preprint \cite{Lim4}.
They are also grateful to Professor Man-Duen Choi and Professor Chi-Kwong Li for helpful discussions.
The research of Huang
was supported by the NSFC grant 11401197 and a Fundamental Research
Funds for the Central Universities.
The research of  Sze was supported by a Hong Kong RGC grant PolyU 502512
and a PolyU central research grant G-YBCR.
This work began when Huang was working as a Postdoctoral Fellow at
The Hong Kong Polytechnic University. He thanks colleagues of HK PolyU for their hospitality and support.

\bigskip\noindent
{\bf Addresses}

\noindent
(Z. Huang) College of Mathematics and Econometrics, Hunan University, Changsha  410082, P.R. China. (Email: mathzejun@gmail.com)

\noindent(S. Shi) Department of Applied Mathematics, The Hong Kong Polytechnic University,
Hung Hom, Hong Kong. (sy.shi@connect.polyu.hk)

\noindent(N.S. Sze) Department of Applied Mathematics, The Hong Kong Polytechnic University,
Hung Hom, Hong Kong. (raymond.sze@polyu.edu.hk)
\end{document}